\documentclass[11pt,english]{smfart}
\usepackage{amsmath, amsthm, amsopn, amsfonts, amssymb}
\usepackage{graphicx}
\usepackage{graphpap}

\usepackage{ifpdf}
\ifpdf
\usepackage{graphicx}  \DeclareGraphicsExtensions{.pdf,.mps}

\else
\usepackage{amsfonts}
\usepackage{graphicx}  \DeclareGraphicsExtensions{.ps,.eps}

\fi

\def\pasdegrille{\let\grille = \pasgrille}
\def\ecriture#1#2{\setbox1=\hbox{#1}
\dimen1= \wd1
\dimen2=\ht1
\dimen3=\dp1
\grille #2 \box1 }
\def\aat#1#2#3{
\divide \dimen1 by 48
\dimen3=\dimen1
\multiply \dimen1 by #1
\advance \dimen1 by -\dimen3
\divide \dimen1 by 101
\multiply \dimen1 by 100
\divide \dimen2 by \count11
\multiply \dimen2 by #2 
\setbox0=\hbox{#3}\ht0=0pt\dp0=0pt
  \rlap{\kern\dimen1 \vbox to0pt{\kern-\dimen2\box0\vss}}\dimen1= \wd1
\dimen2=\ht1}
\def\pasgrille{
\count12= \dimen1 
\divide \count12 by 50
\divide \dimen2 by \count12
\count11 =\dimen2
\ 
\divide \dimen1 by 48
\setlength{\unitlength}{\dimen1}
\smash{\rlap{\ }}
\dimen1= \wd1
\dimen2=\ht1
}
\def\grille{
\count12= \dimen1 
\divide \count12 by 50
\divide \dimen2 by \count12
\count11 =\dimen2
\ 
\divide \dimen1 by 48
\setlength{\unitlength}{\dimen1}
\smash{\rlap{\graphpaper[1](0,0)(50, \count11)}}
\dimen1= \wd1
\dimen2=\ht1
}

\pasdegrille
\usepackage{epic, eepic}

\theoremstyle{plain}

\newtheorem{thm}{Theorem}
\newtheorem{prop}{Proposition}[section]
\newtheorem{cor}[prop]{Corollary}
\newtheorem{lem}[prop]{Lemma}

\newtheorem*{thmcite}{Theorem A}
\theoremstyle{definition}

\newtheorem{rem}{Remark}

\numberwithin{equation}{section}

\title[Global existence]{Global existence for energy
critical waves in $3$-d domains : Neumann boundary conditions}
\author[N. Burq]{Nicolas Burq}
\address{Universit{\'e} Paris Sud,
Math{\'e}matiques,
B{\^a}t 425, 91405 Orsay Cedex, France et Institut Universitaire de France}
\email{Nicolas.burq@math.u-psud.fr}
\author[F. Planchon]{Fabrice Planchon}
\address{ Laboratoire Analyse, G\'eom\'etrie
  \& Applications, UMR 7539, Institut Galil\'ee, Universit\'e Paris
  13, 99 avenue J.B. Cl\'ement, F-93430 Villetaneuse}
\email{fab@math.univ-paris13.fr}

\begin{document}    
\begin{abstract} We prove that the defocusing quintic wave equation,
  with Neumann boundary conditions, is globally wellposed on
  $H^1_N(\Omega) \times L^2( \Omega)$ for any smooth (compact) domain
  $\Omega \subset \mathbb{R}^3$. The proof relies on one hand on $L^p$
  estimates for the spectral projector (\cite{SmSo06}), and on the other hand on
  a precise analysis of the boundary value problem, which turns out to
  be much more delicate than in the case of Dirichlet boundary
  conditions (see~\cite{BLP1}).
\end{abstract}
   
\maketitle
   
\section{Introduction}   
\label{in}
Let $\Omega \in \mathbb{R}^3$ be a smooth bounded domain with
boundary $\partial \Omega$ and $\Delta_N$ the Laplacian acting on
functions with Neumann boundary conditions. In \cite{BLP1} (with G. Lebeau), we derived
Strichartz inequalities for the wave equation from $L^p$ estimates for the associated spectral projector,
obtained recently by H. Smith and C. Sogge~\cite{SmSo06}. As an
application we obtained global well-posedness for the defocusing
energy critical semi-linear wave equation
(with real initial data) in $\Omega $, with Dirichlet boundary
conditions. Here we are interested in the similar question for the
Neumann case,
\begin{equation}\label{eq.NLW}
\begin{aligned}
(\partial_t^2- \Delta) u + u^{5}&=0, \qquad \text{ in } \mathbb{R}_t \times \Omega\\
u _{\mid t=0} = u_0, \qquad \partial_t u_{\mid t=0} &=u_1 , \qquad
\partial_n u
_{\mid \mathbb{R}_t\times \partial \Omega} =0.
\end{aligned}
\end{equation}
Here and thereafter, $\partial_n$ denotes the normal derivative to the
boundary $\partial \Omega$. Equation \eqref{eq.NLW} enjoys the usual conservation of energy
$$ 
E(u)(t) = \int_{\Omega}\Bigl(\frac{ |\nabla u|^2 (t,x) + |\partial_t u|^2(t,x)} 2 + \frac {|u|^{6}(t,x)} {6}\Bigr)dx = E(u) (0)=E_0.
$$
Let $H^1_N$ denote the Sobolev space associated to the Neumann
Laplacian. Our main result reads:
\begin{thm}\label{th.1}
For any $(u_0,u_1) \in H^1_N( \Omega)\times L^2(\Omega)$ there exists
a unique (global in time) solution  $u$ to~\eqref{eq.NLW} in the space
$$ X= C^0( \mathbb{R}_t; H^1_N(\Omega))\cap C^1(\mathbb{R}_t; L^2( \Omega)) \cap L^5_{\text{loc}}(\mathbb{R};L^{10}( \Omega)). $$
\end{thm}
\begin{rem} Smith and Sogge's
  spectral projectors results hold irrespective of Dirichlet or Neumann
  boundary conditions. As such, the local existence result will be
  obtained as in~\cite{BLP1}. However, the non-concentration argument
  requires considerably more care, including non-trivial $L^p$
  estimates for the traces of the solutions on the boundary. We
  provide a self-contained proof of the specific estimates we need,
  which applies to our nonlinear equation. A different set of results in that direction  was announced in \cite{Tataru-trace} in a more general setting for linear wave
  equations.
\end{rem}
\begin{rem} Using the material in this paper, it is rather standard to
  prove existence of global smooth solutions, for smooth initial data
  satisfying compatibility conditions (see
  ~\cite{SmSo95}). Furthermore, our arguments apply equally well to more general defocusing non linearities $f(u)=
  V'(u)$ satisfying
$$
|f(u) | \leq C (1+ |u|^5), \qquad |f'(u)|\leq C (1+ |u|)^4.
$$
Finally, let us remark that our results can be localized
  (in space) and consequently hold also for a non compact domain and in the exterior of any obstacle, and we extend in this framework previous results obtained by Smith and Sogge~\cite{SmSo95} for convex obstacles (and Dirichlet boundary conditions).
\end{rem}
For $s\geq 0$, let us denote by $H^s_N( \Omega)$ the domain of
$(-\Delta_N)^{s/2}$. Finally, , we set $A \lesssim  B$ to mean $A\leq
C B$, where $C$ is an harmless absolute constant (which may change from line to line).
\par \noindent
\section{Local existence}
The local (in time) existence result for~\eqref{eq.NLW} is a direct
consequence of some recent work by Smith and Sogge~\cite{SmSo06}
on the spectral projector defined by $\Pi_\lambda= 1_{\sqrt{-\Delta_N}\in [\lambda, \lambda+1[}$.
\begin{thmcite}[Smith-Sogge~\protect{\cite[Theorem 7.1]{SmSo06}}]\label{th.SmSo} 
Let $\Omega \in \mathbb{R}^3$ be a smooth bounded domain and $\Delta$ be the Laplace operator with Dirichlet or Neumann boundary conditions, then there exists $C>0$ such that for any $\lambda \geq 0$,
\begin{equation}\label{eq.proj}
 \|1_{\sqrt{-\Delta}\in [\lambda, \lambda+1[} u\|_{L^5(\Omega)}\lesssim  \lambda^{\frac 2 5} \|u\|_{L^2(
 \Omega)}.
\end{equation}
\end{thmcite}
From which we derive Strichartz estimates, which are
optimal w.r.t scaling.
\begin{thm}\label{th.2}
Assume that for some $2 \leq q < + \infty$, the spectral projector $\Pi_\lambda$ satisfies
\begin{equation}\label{eq.estproj}
 \|\Pi_\lambda u\|_{L^q(\Omega)}\lesssim  \lambda^{\delta} \|u\|_{L^(
 \Omega)}.
\end{equation}
Then the solution to the wave equation $v(t,x) = e^{it \sqrt{-
\Delta_N}}u_0$ satisfies
$$
 \|v\|_{L^q((0,2\pi)_t \times \Omega_x)}\lesssim  \|u_0 \|_{H_N^{\delta+\frac 1 2 - \frac 1 q}}.
$$
\end{thm}
The proof given in \cite {BLP1} applies verbatim and we therefore skip
it.

  \begin{prop}
\label{propex}
    If $u,f_1,f_2$ satisfy
$$
 (\partial_t^2- \Delta)u =f_1+f_2, \quad \partial_n u_{\mid \partial \Omega} =0, \quad u _{\mid t=0}= u_0, \quad \partial_t u _{\mid t=0} = u_1
$$
then
\begin{multline}\label{eq.besov}
 \|u\|_{L^5((0,1);W^{\frac 3 {10},5}_N ( \Omega))}+
 \|u\|_{C^0((0,1); H^1_N( \Omega))}+\|\partial_t
 u\|_{C^0((0,1); L^2( \Omega))}\\
  \leq C
 \left(\|u_0\|_{H^1_N( \Omega)} + \|u_1\|_{L^2( \Omega)}+\|f_1\|_{L^1{((0,1);L^2(\Omega))}}+
 \|f_2\|_{L^{\frac 5 4}((0,1); W^{\frac 7 {10},\frac 5 4} ( \Omega))}\right)\,.
\end{multline}
Furthermore, \eqref{eq.besov} holds (with the same constant $C$) if one replaces the time interval $(0,1)$ by any interval of length smaller than $1$.
  \end{prop}
\begin{rem} Notice that by Sobolev embedding and trace lemma, one has
\begin{equation}\label{eq.faible}
 \|u\|_{L^5((0,1);L^{10}( \Omega))}+ \|u_{\mid \partial\Omega}\|_{L^5((0,1); L^{\frac {20}
  3}( \partial \Omega))} \lesssim  \|u\|_{L^5((0,1);W^{\frac 3 {10},5}_N ( \Omega))},
 \end{equation}
which will play a crucial role in the non-concentration argument.
\end{rem}
\begin{cor}\label{cor.1} 
For any initial data $(u_0, u_1) \in H^1_N(\Omega) \times L^2(\Omega)$, the critical non linear wave equation~\eqref{eq.NLW} is locally well posed in 
$$
 X_T= C^0([0,T]; H^1_N( \Omega)) \cap L^5_{\mathrm{loc}}((0,T); L^{10}( \Omega)) \times
 C^0([0,T]; L^2( \Omega))$$ (globally for small norm initial data).
\end{cor}
We refer to Appendix A for a proof of Proposition \ref{propex}. The
 proof of Corollary~\ref{cor.1} proceeds by a standard fixed point argument with $(u, \partial_t u)$ in the space $X_T$
 with a sufficiently small $T$ (depending
 on the initial data $(u_0, u_1)$). Note that this local in time
 result holds irrespective of the sign of the nonlinearity.

Finally, to obtain the global well posedness result for small initial data, it is enough to remark that if the norm of the initial data is small enough, then the fixed point can be performed in $X_{T=1}$.
Then the control of the $H^1$ norm by the energy (which is conserved along the evolution) allows to iterate this argument indefinitely leading to global existence. Note that this
 result holds also irrespective of the sign of the nonlinearity because for small $H^1$ norms, the energy always control the $H^1$ norm. 
\section{Global existence}
It turns out that our Strichartz estimates are strong enough to extend
local to {\em global} existence for arbitrary (finite energy) data,
when combined with trace estimates
and non concentration arguments.

Before going into details, let us sketch the proof.

Remark that if $f= u^5$, we can estimate
\begin{equation}
\begin{aligned}
\|u^5\|_{L^{\frac 5 4}((0,1);L^{\frac {30} {17}}( \Omega))}&\leq  \|u\|^4_{L^5((0,1); L^{10}( \Omega))} \|u \|_{L^\infty((0,1); L^6( \Omega))}\,,\\
\|\nabla_x (u^5)\|_{L^{\frac 5 4}((0,1);L^{\frac {10} {9}}( \Omega))}&=5\|u^4\nabla_x u\|_{L^{\frac 5 4}((0,1);L^{\frac {10} {9}}( \Omega))}\\
&\leq 5 \|u\|^4_{L^5((0,1); L^{10}( \Omega))} \|u \|_{L^\infty((0,1); H^1( \Omega))}\,.
\end{aligned}
\end{equation}
Interpolating between these two inequalities yields
\begin{equation}\label{eq.bonstrichartz} 
\|u^5\|_{L^{ \frac 5 4}((0,1); W^{\frac 7 {10},\frac 5 4}(
  \Omega))}\lesssim  \|u\|^4_{L^5((0,1); L^{10}( \Omega))} \|u \|^{\frac 3
  {10}}_{L^\infty((0,1); L^6( \Omega))}\|u \|^{\frac 7
  {10}}_{L^\infty((0,1); H^1( \Omega))}\,.
\end{equation}
Following ideas of Struwe~\cite{St88}, Grillakis~\cite{Gr90} and
Shatah-Struwe~\cite{ShSt93, ShSt94}, we will localize these estimates
on small light cones and use the fact that the $L^\infty_t;L^6_x$ norm
is small in such small cones. Unlike with Dirichlet boundary
conditions, we cannot hope to have a good control of the $H^1$ norm of the trace of
$u$ on the boundary (this control is known to be false even for the linear problem), and it will require a more delicate argument to
handle boundary terms.
\subsection{The $L^6$ estimate}
In this section we shall always consider  solutions to~\eqref{eq.NLW} in 
\begin{equation}\label{eq.espace}
 X_{<t_0}= C^0([0,t_0); H^1_N( \Omega)) \cap L^5_{loc}([0,t_0); L^{10}( \Omega)) \times
 C^0([0,t_0); L^2( \Omega)).
\end{equation}
Moreover, we assume these solutions to have energy bounded by a fixed
constant (namely, $E_0$). By a standard procedure, such solutions are obtained as limits in $X_{<t_0}$ of smooth solutions to the analog
of~\eqref{eq.NLW} where the nonlinearity and the initial data have been smoothed
out. Consequently all the subsequent integrations by parts will be
licit by a limiting argument. 

\subsubsection{The flux identity} By time translation, we shall assume later that $t_0=0$.
 Let us first define
$$
 Q = \frac{|\partial_tu|^2 + |\nabla_xu|^2 } 2 + \frac {|u|^6} 6 + \partial _t u (\frac x t \cdot \nabla_x) u,
$$
\begin{equation}\label{eq.P} 
P = \frac x t \Bigl( \frac{|\partial_tu|^2 - |\nabla_xu|^2 } 2 - \frac {|u|^6} 6\Bigr)
+ \nabla_x u  \Bigl( \partial_t u + (\frac x t \cdot \nabla_x) u + \frac u t \Bigr)\,,
\end{equation}
\begin{gather*}
  \Omega_S^T  =  [S,T]\times \Omega, \qquad \partial \Omega_S^T  =  [S,T]\times \partial \Omega\\
   D_T  = \{x; |x| <-T\},\qquad M_S^T  =\{x;|x|=-t\}\cap \Omega_S^T \\
 K_S^T  = \{(x,t); |x|<-t\}\cap \Omega_S^T\\
 \partial K_S^T  =  (\partial \Omega_S^T\cap K_S^T) \cup D_T \cup D_S \cup M_S^T\\
e(u)  = \Bigl(\frac{ |\partial_t u|^2 + |\nabla_x u|^2 } 2 + \frac  {|u|^6} 6, - \partial_t u \nabla _x u\Bigr) 
\end{gather*}
and the flux across $M_S^T$
$$
\text{ Flux }(u, M_S^T)= \int_{M_S^T} \langle e(u), \nu\rangle   \,d\rho(x,t),
$$
where 
\begin{equation}
  \label{eq:normalcone}
\nu=\frac 1{ \sqrt 2 |x|} (-t,-\frac{x} t)=\frac 1{ \sqrt 2 |x|} (|x|,\frac{x} {|x|}) 
\end{equation}
 is the outward normal to $ M_S^T$ and $d\rho(x,t)$ the induced measure on $M_S^T$. Remark that 
\begin{align}\label{eq.fluxest}
\text{ Flux }(u, M_S^T) & =  \int_{M_S^T}\frac{|\partial_t u|^2 + |\nabla_x u|^2} 2 + \frac {|u|^6} 6 - \partial_t u \frac {x} { |x|} \cdot \nabla_x u \,d \rho(x,t)\\
 & = \int_{M_S^T}\frac1 2 |\frac x {|x|} \partial_t u - \nabla_x u|^2
 + \frac {|u|^6} 6 \,d\rho(x,t)\geq 0 \,.\nonumber
\end{align}
One may notice that the boundary of $K_S^T$ is $M_S^T\cup
\partial\Omega_S^T\cup D_S\cup D_T$, and that, due to the boundary
condition, there is no flux through $\partial\Omega_S^T$.

An integration by parts gives (see Rauch~\cite{Ra81} or~\cite[(3.3')]{SmSo95})
\begin{multline}\label{eq.locener}
\int_{x\in \Omega, |x|<-T}\Bigl(\frac{ |\partial_t u|^2 + |\nabla_x u|^2 } 2 + \frac  {|u|^6} 6\Bigr)(x,T) dx + \text{ Flux } (u, M_S^T) \\
=\int_{x\in \Omega, |x|<-S}\Bigl(\frac{ |\partial_t u|^2 + |\nabla_x u|^2 } 2 + \frac  {|u|^6} 6\Bigr)(x,S) dx= E_{loc}(S).
\end{multline}
Therefore $u_{\mid M_S^T}$ is bounded in $H^1(M_S^T)\cap
L^6(M_S^T)$ (uniformly with respect to $T<0$), and $E_{loc}(S)$
being a non-negative non-increasing function, it has a limit when
$S\rightarrow 0^-$, and
\begin{equation}\label{eq.fluxestbis}
\text{Flux}(u, M_S^0)=\lim_{T\rightarrow 0^-}\text{Flux}(u, M_S^T)
=\lim_{T\rightarrow 0^-}( E_{loc} (S) - E_{loc}(T))
\end{equation} exists and satisfies      
\begin{equation}\label{eq.flux}
\lim_{S\rightarrow 0^-}\text{ Flux } (u, M_S^0)=0.
\end{equation}
\subsubsection{A priori estimate for traces}
We prove an priori estimate on finite energy solutions
of~\eqref{eq.NLW}, which is reminiscent of results obtained in
\cite{Tataru-trace} for variable coefficients linear wave
equations. Observe that, unlike in the Dirichlet case, we don't have
any uniform Lopatinskii condition, which prevents control of the
gradient on the boundary.
The following result will provide a substitute. It shows that even though (due to
the failure of the uniform Lopatinskii condition), the $H^1$ norm of
the trace is known to be unbounded (at least for the linear equation),
the integral on the boundary of a specific quadratic form (namely the
so-called $Q_0$ null form) remain bounded (see~\cite{Tataru-trace} for related
results). 
\begin{prop}\label{prop.apriori2} Let $n(x)$ be the exterior unit normal vector to
$\partial\Omega$, $-\varepsilon<S<T\leq 0$, $d\sigma$ the induced
measure on $\partial\Omega$, and $u$ a solution to \eqref{eq.NLW}. Then
\begin{equation} \label{eq.apriori2}
\left |\int_{K^T_S\cap  \partial\Omega_S^T} (|\partial_t u|^2-|\nabla
u|^2-\frac{u^6} 3)\,n(x)\cdot x \,d\sigma dt\right| \lesssim  |S|^2 (E_0+E_0^{\frac 2 3}),
\end{equation}
where the tip of the cone $K^0_{S}$ is located on $\partial\Omega$ and
has been set to be $0\in \partial \Omega$.
\end{prop} 
\begin{rem} 
For the solution of the {\em linear} wave equation, $\Box u =0$, taking the trace on the boundary $\partial \Omega$ in~\eqref{eq.besov} gives 
\begin{multline}\label{eq.besovbis}
 \|u\mid _{\partial\Omega}\|_{L^5((0,1);W^{\frac 1 {10},5} ( \partial \Omega))}+ \|u\mid _{\partial\Omega}\|_{L^\infty((0,1); H^{1/2} ( \partial\Omega))}\\
\leq C\|u\|_{L^5((0,1);W^{\frac 3 {10},5}_N ( \Omega))}+ \|u\|_{L^\infty((0,1); H^{1} ( \Omega))}
  \leq C
 \left(\|u_0\|_{H^1_N( \Omega)} + \|u_1\|_{L^2( \Omega)}\right)
\end{multline} 
And interpolating between the two estimates in the l.h.s. of~\eqref{eq.besovbis} gives

$$\|u\mid _{\partial\Omega}\|_{L^6((0,1)\times \partial \Omega)}
  \leq C
 \left(\|u_0\|_{H^1_N( \Omega)} + \|u_1\|_{L^2( \Omega)} \right)
$$
As a consequence, one could infer that the $u^6$ term in the r.h.s. of~\eqref{eq.apriori2} is not the most important part of the estimate. However, we will not use this fact here and shall prove the estimate~\eqref{eq.apriori2} as a whole
\end{rem}

\begin{proof} 
The main step in the proof of Proposition~\ref{prop.apriori2} is an
integration by parts using suitable weight and vector field.
\begin{lem}
There exists a smooth vector field $Z$ (depending on $(t,x)$) defined
on a neighborhood of $x_0$ in $[-t_0,0)\times \overline \Omega$ such that 
\begin{enumerate}
\item We have $Z_{\mid\partial\Omega}= \partial_n +\tau$ where $\partial_n$ is the interior normal derivative to the boundary
and $\tau$ is a vector field which is tangent to the boundary;
\item  The restriction of $Z$ to the sphere $S_t=\{x\in \Omega; |x| = -t\}$ is tangent to that sphere. 
\end{enumerate}
\end{lem}
\begin{proof}
 Performing a linear orthogonal change of variables, we can assume that $\partial_3=\partial_{n(0)}$, with $n(0)$
the normal vector to
$\partial\Omega$ at $x=0$. We now build a
vector field $Z \in C^\infty ( \Omega; T\Omega)$ (as such, it does not
have a $\partial_t$ component) such that its restriction to
$\partial\Omega$ is close to $\partial_n$ (which, in a
neighborhood of $x=0$, $B(0,\eta)$, is essentially $\partial_3$ up to
an $O(\eta)$ term).

Consider the half-sphere $x_3=\sqrt{1-r^2}$, with
$r^2=x_1^2+x_2^2$. To define a vector
field $X$ in
the zone $x_3<\frac 3 4$ (thus multiplying by a cut-off $\phi(x)$ such
that $\phi=0$ for $x>3/4$ and $\phi=1$ for $x<1/4$), we set
$X=X_1=\partial_3$ for
$r<1/4$; for $r>3/4$, we set
$X=X_2=\partial_3-r^{-1} x_3 \partial_r$. On the half-sphere, $X$ will be
tangent and pointing in the same direction as $\partial_3$ (or $\partial_\nu$). Now, we may
smoothly connect both zones and define $X$ within the entire
half-sphere: consider $\phi \in C^\infty_0 (- 1/2, 1/2)$ equal to $1$ on $(0, 1/4)$ and 
$$
X=\phi(x_3)(\phi(r)\partial_3 +(1-\phi(r))(\partial_3-\frac{x_3}{r} \partial_r)).
$$
Then one rescales $(x_3,r)$ to be $(x_3/t,r/t)$. Note that on the $x_3=0$
axis, $X=\partial_3$. We now set
$$
X_t=\phi(x_3/t)(\phi(r/t) X_1(x/t)+(1-\phi(r/t)) X_2)(x/t).
$$
Remark that $X_1,X_2$, as previously defined, are time independent (and smooth in zones where
$X\neq 0$). The vector field $X_t$ is tangent to the sphere $S_t$ by
definition. However, it requires to be renormalized so that its
component on the normal to the spatial boundary is one: define $z(x)\in
\partial\Omega$ to be the orthogonal projection to the (spatial) boundary, we set
$$
Z=\frac{1}{X_{t}(z(x)) \cdot n(z(x))} X_t,
$$
which is easily seen to satisfy all our requirements.
\end{proof}

Now we define a smooth weight $w(x)$, such that its restriction to the
boundary $\partial \Omega$ is $x\cdot n(x)$, and such that
$|w(x)|=O(|r|^2)$: $w(x)=z(x)\cdot n(z(x))$ is such a weight.

To prove Proposition \ref{prop.apriori2}, we start with the following commutator estimate.
\begin{prop} 
\label{marre}
There exists $C>0$ such that 
\begin{equation}
   \left|\int_{K^T_S\cap \Omega_S^T} [(\partial_t^2- \Delta),Z]u(t,x)\cdot
   {u} (t,x) w(x) dx dt\right| \lesssim  |S|^2 (E_0+E_0^{\frac 2 3}).
\end{equation}
\end{prop}
\begin{proof} We shall make a repeated use of the following
\begin{lem}\label{prop.restes1}
Consider a function $a(x,t)$ such that
$$
 \forall (x,t); |x| \leq -t, t \in (-1,0), \quad |a(x,t)| \leq M.
$$
Then there exists $C>0$ such that for any  solution with energy $E_0$,
for all $S,T$ s.t. $-1\leq S <T<0$,
\begin{equation} 
 \Bigl| \int_{K_S^T}a(x,t) u(x,t)
\nabla_{x,t} u(x,t)\,dx dt \Bigr| \lesssim  E^\frac 2 3_0 |S|^2.
\end{equation}
\end{lem}
\begin{proof} We have
\begin{multline}
\Bigl| \int_{K_S^T}a(x,t) u(x,t) \nabla_{x,t} u(x,t) \,dx dt \Bigr| \\
\begin{aligned}
&\leq \Bigl(\int_S^T |\nabla_{x,t} u(x,t)|^2 dx dt\Bigr)^{1/2} \Bigl( \int_{K_S^T} | u(x,t)|^6 dx dt \Bigr)^{1/6}\Bigl( \int_{K_S^T} | a(x,t)|^3 dx dt \Bigr)^{1/3}\\
&\lesssim  E^{\frac 1 2+\frac 1 6 }(u) |S|^{\frac 1 2+\frac 1 6}\Bigl( \int_{K_S^T} | a(x,t)|^3 dx dt \Bigr)^{1/3} \\
&\lesssim  (E_0|S|)^{2/3} \Bigl( \int_S^T\int_{|x|<-t} dx dt
\Bigr)^{1/3}= C E^\frac 2 3_0 |S|^2,
\end{aligned}
\end{multline}
which achieves the proof.
\end{proof}
We now prove a second lemma which will immediately yield Proposition \ref{marre}:
\begin{lem}\label{prop.restes2}
Assume that $Q=\sum_{ij} q_{ij}\partial^2_{ij}$ is a smooth (space-time) second order operator with real
coefficients $(q_{ij})_{ij}$, $0\leq i,j\leq 3$, satisfying
\begin{equation}
  \label{eq:coeff}
  \forall i,j,\,\forall (x,t); |x| \leq -t, t \in
(-1,0), \quad |\partial_{x,t}^\alpha q_{ij}(x,t) |\lesssim  {|t|^{1-|\alpha|}}, \qquad |\alpha | \leq 2.
\end{equation}
Moreover, assume that (recalling \eqref{eq:normalcone} for the
definition of $\nu$)
\begin{equation}
  \label{eq:marretangent}
  T=\sum_{i,j} ((\nu \cdot \partial_i) q_{ij}) \partial_j \text{ is a vector field which is
tangent to } M_S^T;
\end{equation}
 then, for any solution with energy $E_0$, and any $ -1 \leq S <T<0$,
\begin{equation} 
 \Bigl| \int_{K_S^T}Q(x,t,D_{x,t}) u(x,t)
u(x,t)dx dt \Bigr| \lesssim  |S|^2 (E_0+E_0^\frac 2 3).
\end{equation}
\end{lem}
\begin{proof}
We integrate by parts
\begin{multline}\label{eq.ipp}
\int_{K_S^T}Q(x,t,D_{x,t}) u(x,t) u(x,t)dx dt\\
\begin{aligned}
&= - \sum_{i,j} \int_{K_S^T}\partial_i (q_{i,j}(x,t))\partial_j u(x,t)  u(x,t)dx dt \\
&\qquad- \sum_{i,j} \int_{K_S^T}q_{i,j}(x,t)\partial_j u(x,t) \partial_i ( u(x,t))dx dt \\
&\qquad+\sum_{i,j}\Bigl[\int _{D_T}q_{0,j}(x,t)\partial_j u(x,t) u(x,t)) dx -\int _{D_S}q_{0,j}(x,t)\partial_j u(x,t) u(x,t)) dx\Bigr]\\
&\qquad+ \sum_{i,j}\int _{M_S^T}(\nu(x,t)\cdot \partial_i) q_{i,j}(x,t)\partial_j u(x,t) u(x,t)) d\rho \\
&\qquad+\sum_{i,j}\int _{K_S^T \cap \partial \Omega_S^T}
(-n(x)\cdot\partial_i) q_{i,j}(x,t)\partial_j u(x,t) u(x,t)) d\sigma dt\\
&= I+II+III+IV+V,
\end{aligned}
\end{multline} 
where we recall $\nu(x, t) =\frac{1}{\sqrt 2} (\partial_r + \partial
_t)$ to be the outward normal vector to $M_S^T$ and $n(x)$ to be the
inward normal vector to $\partial \Omega$, while $d\rho$
(resp. $d\sigma$) is the induced measure on $M_S^T$ (resp. $\partial\Omega$).

The contribution of $I$ is dealt with using
Lemma~\ref{prop.restes1}. Next, 
$$
II\lesssim \|q_{i,j}\|_{L^\infty(K_S^T)}|S| E(u)\lesssim  |S|^2 E_0.
$$
The contribution of $D_S$ in $III$ is bounded (using H\"older inequality) by 
\begin{multline}
 \Bigl(\int_{D_S} |\nabla_{x,t} u(x,t)|^2 dx dt\Bigr)^{1/2} \Bigl( \int_{D_S} | u(x,t)|^6 dx dt \Bigr)^{1/6}\Bigl( \int_{D_S} | q_{i,j}(x,t)|^3 dx  \Bigr)^{1/3}\\
\lesssim  E(u)^{2/3} \Bigl( \int_{D_S} | q_{i,j}(x,S)|^3 dx \Bigr)^{1/3} 
\lesssim  E_0^{2/3} |S|\Bigl( \int_{|x|<|s|} dx  \Bigr)^{1/3}= C E_0^{2/3} |S|^2.
\end{multline}
We deal with the contribution of $D_T$ to $III$ similarly. To bound
the contribution of $IV$, we remark that according to our
assumptions on $Q$, the vector field 
$$
\frac 1
{\sqrt{2}} T = \sum_{i,j} (\nu(x,t)\cdot \partial_i)
q_{i,j}(x,t)\partial_j
$$
is tangential to $M_S^T$. As a consequence (using H\"older inequality),
\begin{multline}
\label{3.18}
IV\lesssim  \Bigl(\int_{M_S^T} |\nabla_{tan} u(x,t)|^2 d\rho \Bigr)^{1/2} \Bigl(
\int_{M_S^T} | u(x,t)|^6 d\rho \Bigr)^{1/6}\Bigl( \int_{M_S^T} |
q_{i,j}(x,t)|^3 d\rho  \Bigr)^{1/3}\\
\lesssim  (\text{Flux}(u, M_S^T))^{2/3} \Bigl( \int_{M_S^T} | q_{i,j}(x,t)|^3 d\rho \Bigr)^{1/3} \\
\lesssim  (\text{Flux}(u, M_S^T))^{2/3}|S|^2\lesssim  E_0^{\frac 2 3} |S|^2
\end{multline}
where in the last inequality we used~\eqref{eq.locener}.

 It remains to bound the contributions of $V$ in the right hand side of~\eqref{eq.ipp}. Using the Neumann boundary condition, we can replace the vector fields 
$\partial_j$ by their tangential components $\partial_j - (\partial_j
\cdot n(x)) n(x)$. In a coordinate system $y$ in the boundary and
abusing notation for $\partial_{y_i}$, we now have to compute 
\begin{equation}
V=\int_{K_S^T\cap \partial \Omega}\widetilde {q}_{i,j}(y,t)\partial_j u(y,t) u(y,t)) \,dy dt\end{equation}
with $\widetilde {q}$ satisfying the same estimates as $q$, namely \eqref{eq:coeff}. Integrating by parts gives
\begin{equation}
\label{eq.5.1}
 V= \frac 1 2 \int_{K_S^T\cap \partial \Omega_S^T}\partial_j(\widetilde
 {q}_{i,j}(y,t)) |u(y,t)|^2 \,d\sigma dt + \frac 1 2 \int_{M_S^T\cap
 \partial \Omega_S^T}\widetilde {q}_{i,j}(y,t) |u(y,t)|^2 \,d\rho= V.1+V.2,
\end{equation}
where the first term in the right hand side appears when the
derivative hits the coefficients while the second term comes from the
contribution of the boundary of $K_S^T\cap \partial \Omega_S^T$). 

To conclude, we use the following
\begin{lem}\label{lem.restr}
For any solution $u$ and any $S<T<0$, we have
\begin{equation}\label{eq.bordbord}
 \|u_{\mid M_S^T\cap \partial \Omega_S^T}\|^2_{L^4(M_S^T\cap \partial
 \Omega_S^T)}\lesssim  \text{ Flux } (u, M_S^T)\lesssim  E_0.
\end{equation}
\end{lem}
\begin{proof}
 Let us first conclude the proof
of Lemma~\ref{prop.restes2} assuming Lemma \ref{lem.restr}. To deal with the first term in the
right hand side of~\eqref{eq.5.1}, we use  H\"older inequality and
obtain (using the trace theorem and then Sobolev embedding on $\partial\Omega$)
\begin{multline}
|V.1| \lesssim  \Bigl(\int_{\partial \Omega_S^T}|u|^4 d\sigma
 dt\Bigr)^{\frac 1 2}\Bigl(\int_{K_S^T\cap \partial \Omega_S^T}
 |\partial_j(\widetilde {q}_{i,j}(x,t))|^{2}\Bigr)^{\frac 1 2}\\
\lesssim |S|^{\frac 3 2 +\frac 1 2} \sup_t \Bigl(\int_{\partial \Omega}|u|^4 d\sigma
 \Bigr)^{\frac 1 2}\lesssim  |S|^2 E_0.
\end{multline}
To deal with the second term in the right hand side of~\eqref{eq.5.1}, we also use  H\"older inequality and obtain (using Lemma~\ref{lem.restr})
\begin{equation}
|V.2| \lesssim  \Bigl(\int_{M_S^T\cap \partial \Omega_S^T} |u(x,t)|^4
 d\sigma\Bigr)^{1/2}\Bigl(\int_{M_S^T\cap \partial \Omega_S^T}
 |(\widetilde {q}_{i,j}(x,t))|^{2}\Bigr)^{1/2}\lesssim |S|^2 E_0.
\end{equation}
which ends the proof of Lemma~\ref{prop.restes2}\end{proof}
Let us now prove Lemma~\ref{lem.restr}. The idea is roughly to take
benefit of the fact that the flux controls the $H^1$ norm on the cone
$M_S^T$; consequently by trace theorems it controls the $H^{1/2}$
norm of the restriction of $u$ on the intersection of the cone and the
spatial boundary. Hence, by
Sobolev's theorem, it controls the $L^4$ norm. However, the geometry of the cone
becomes singular near $t=0$ and we have to be careful when implementing
this idea. For this we decompose the time interval $(S,T)$ into a
union of dyadic intervals
$$ (S,T)= \cup_{j=j_0}^{j_1-1} ( t_j,t_{j+1}), \quad t_j = -2^{-j} $$
 (assuming for simplicity that $S= - 2^{-j_0}$ and $T= -2^{-j_1}$), and we decompose the integral~\eqref{eq.bordbord} into
$$
 \sum_{j=j_0}^{j_1-1} \|u_{\mid M_S^T\cap \partial
  \Omega_S^T}\|^4_{L^4(M_S^T\cap \partial \Omega_S^T)}
$$
For each integral, we perform the change of variables 
$$ (t,x) \mapsto (s= 2^{j}t, y = 2^j x), u_j(s,y) = u(2^js, 2^jy)$$
and we have to bound for any $j_0 \leq j \leq j_1 -1$
$$ \|(u_j)_{\mid2^j (M_{-2^{-j}}^{-2^{-j-1}}\cap \partial \Omega_{-2^{-j}}^{-2^{-j-1}})}\|^4_{L^4}$$
Observe now that 
$$2^j (M_{-2^{-j-1}}^{-2^{-j-2}}\cap \partial \Omega_{-2^{-j-1}}^{-2^{-j-2}})$$ is a smooth
family of hypersurfaces of the cone  (included in the smooth part of
the cone corresponding to  $\{t\in (1, 1/2)\}$ and consequently, the
trace theorem applies with uniform constants; tracking the change
of variables yields
$$
 \|u_{ \mid(M_{-2^{-j}}^{-2^{-j-1}}\cap \partial
 \Omega_{-2^{-j}}^{-2^{-j-1}})}\|^4_{L^4}\lesssim  \|u _{\mid (M_{-2^{-j-1}}^{-2^{-j-2}}\cap
 \partial \Omega_{-2^{-j-1}}^{-2^{-j-2}})}\|^4_{H^1}.
$$
Summing the pieces back provides the desired estimate.
\end{proof}
Back to proving Proposition \ref{marre}, one may easily see that the
coefficients of the second order operator $w(x) [\Box,Z]$ satisfy the
decay conditions of Lemma~\ref{prop.restes1} and \ref{prop.restes2}. We
are left with \eqref{eq:marretangent}: from decomposing the Laplacian in polar
coordinates, the first order term ($r^{-1}\partial_r$) is harmless, and so is the
angular part. Then, if
$A=\partial^2_t-\partial^2_r=(\partial_t+\partial_r)(\partial_t-\partial_r)$,
$w(x) [A,Z]$ doesn't have a $ (\partial_t+\partial_r)^2$ component,
hence it satisfies \eqref{eq:marretangent}.
\end{proof}

Having proved that
\begin{equation}\label{eq.2.3}
 \Bigl|\int_{K^T_S\cap \Omega} [(\partial_t^2- \Delta), Z] u(t,x)
 {u} (t,x) w(x) dx dt\Bigr| \lesssim  S^2 (E_0+E_0^{\frac 2 3}),
\end{equation}
we may return to the proof of Proposition \ref{prop.apriori2} and perform integration by parts with the
$\Box$ operator in \eqref{eq.2.3}. Denote by $D=\Omega_S^T\cap K_S^T$ the
space-time domain of integration. Its boundary $\partial D$ will be
the reunion of two time slices, $D_T\cap \Omega$ and $D_S\cap \Omega$,
one space slice (on the boundary) $\partial\Omega_S^T\cap K_S^T$  and the cone
boundary inside the domain, $\Omega_S^T\cap M_S^T$. 

As
\begin{multline}\label{eq.2.4}
 \int_{D} [(\partial_t^2- \Delta), Z] u(t,x) \cdot
 u(t,x) w (x)dx dt =  \int_{D} Z (u^5) u (t,x)
 w(x) dx dt\\ 
 {}+\int_{D} (\partial^2_t-\Delta) Z
 (u) \cdot u (t,x) w(x) dx dt,
\end{multline}
Leaving aside the first term, we compute the remaining term, namely
$$
\int_{D} w( x, t) u
\Box Z u \,dxdt.
$$
We will
apply Green's second identity: recall that over 4-dimensional domains
$V$, with $\nu_{\partial V}$ the outward normal to the boundary,
\begin{equation}
  \label{eq:greenid}
  \int_V \phi \Box \psi-\psi \Box \phi =\int_{\partial V}\left( \psi\begin{pmatrix} -\partial_t
  \phi\\ \nabla \phi\end{pmatrix}\cdot \nu_{\partial V}- \phi\begin{pmatrix} -\partial_t
  \psi\\ \nabla \psi\end{pmatrix}\cdot \nu_{\partial V} \right)\,.
\end{equation}
Let us start with $Zu \Box(uw)= Zu (u\Box w+w \Box u+Q_0(w,u))$: The last term is controlled by $O(S^2) E(u)$: in fact,
one has 2 derivatives hitting the $u$'s (one from $Z$ and one from
$Q_0$), and the remaining one from $Q_0$ on $w$, which
loses an $S$; and then integration over $t$ which regains the lost
factor $S$. The first term is no worse: one loses both factors $S$ in
deriving $w$ twice, but on the other hand we get
$$
\int_{D_t} u^2 dx \leq (\int_{D_t} dx)^{\frac 2 3} (\int_{D_t} u^6
dx)^{\frac 1 3}\lesssim  S^2 E(u)^{\frac 1 3},
$$
since the volume of the ball $D_t$ is $t^3$ which produces the $S^2$
factor. The middle term is, after substitution,
$$
\int_D w(x)Zu(x,t) (-u^5) \,dxdt,
$$
which may be added to the term we left in \eqref{eq.2.4}, to get
\begin{multline}\label{eq.2.1}
4 \int_{D} (Z u) u^5 w(x) \,dx dt = \frac 2 3 \int_D Z(u^6)
w(x)\,dxdt\\
=-\frac 2 3 \int_D u^6 Z(w)\,dxdt-\frac 2 3\int_{ \partial\Omega_S^T\cap K_S^T}
u^6 w d\sigma,
\end{multline}
as $Z$ is tangent to both the time slices and the cone (hence, the
boundary contributions of these regions vanish). Recalling that $Z(w)=O(S)$,
$$
 \Bigl|\int_{K^T_S\times \Omega} u^6 Z w(x) \,dx dt\Bigr|  \lesssim 
 S^2 E(u),
$$
and we collect a term
\begin{equation}
  \label{eq:potentieltrace}
  -\frac 2 3\int_{ \partial\Omega_S^T\cap K_S^T}
u^6 w d\sigma,
\end{equation}
for later use.

We are now left with a space-time boundary term $J$ coming from our
application of \eqref{eq:greenid}, which is a difference of
two terms
\begin{equation}
  \label{eq:TB}
  J=\int_{\partial D} (u w)\,  N \cdot \begin{pmatrix} \partial_t
  \\-\nabla\end{pmatrix} Zu - Z u\,  N \cdot \begin{pmatrix} \partial_t
  \\-\nabla\end{pmatrix} (uw)=J_1-J_2\,.
\end{equation}
Here $N$ is the outward normal derivative to the boundary
$\partial D$. The second
term $J_2$ splits itself in three sub-terms.
\begin{itemize}
\item The first two are boundary terms on $M_S^T\cap \Omega$ and
  $D_S\cup D_T$. The cone term is controlled by $S^2$ times the flux, as one has
  only tangential derivatives on $u$ (either $Z$ or
  $L={\sqrt 2}^{-1}(\partial_t-\partial_r)$). The time-slice terms are similarly
  controlled by $S^2$ times the energy.
\item The last one is
$$
\int_{\partial\Omega_S^T \cap K_S^T} u Zu \partial_n w + w
Zu \partial_n u \, d\sigma dt,
$$
and both terms vanish: the first one because $\partial_n w=0$ and the
second one because of the Neumann boundary condition.
\end{itemize}
We are thus left with $J_1$, which we split again in
three terms
\begin{eqnarray*}
  J_1 & = & \int_{M_S^T} w u\, (L
  Z u) \, d\rho \\
 & & {}+ \int_{D_T\cap \Omega} w u\, \partial_t Z u \, dx -
  \int_{D_S\cap \Omega} w u \,\partial_t Z u \, dx \\
 & & {}+\int_{K_S^T\cap \partial\Omega_S^T} (\partial_n Z u) \,u w
  \,d\sigma dt\\
 & = & K_1+K_2+K_3.
\end{eqnarray*}
Consider $K_3$: we chose $Z$ such that $Z_{\mid
  \partial\Omega}=\partial_n+\tau$, where $\tau$ is a tangent vector field
  to $\partial\Omega$, so that
\begin{eqnarray*}
  K_3 & = & \int_{K_S^T\cap \partial\Omega_S^T}uw \, \partial^2_n
    u+w u [\partial_n,\tau] u\,d\sigma dt\\
    & = & \int_{K_S^T\cap \partial\Omega_S^T}uw \,
    (\partial^2_t-\Delta_\mathrm{tan}) u+ u^6 w \,d\sigma dt+\int_{K_S^T\cap \partial\Omega_S^T}  w u [\partial_n,\tau] u\,d\sigma dt.
\end{eqnarray*}
We integrate by parts, to get
\begin{multline*}
   \int_{K_S^T\cap \partial\Omega_S^T}uw \,
    (\partial^2_t-\Delta_\mathrm{tan}+u^4) u\,d\sigma dt= \int_{K_S^T\cap
 \partial\Omega_S^T}(|\nabla_\mathrm{tan} u|^2-|\partial_t u|^2+u^6)
 w\,d\sigma dt\\{}+\int_{K_S^T\cap
 \partial\Omega_S^T}u \partial_{\mathrm{tan}} u(\eta w+\partial_{\mathrm{tan}}
 w\,d\sigma dt+B,
\end{multline*}
where $B$ is the boundary term to be specified below. The second term can
be dealt with as in \eqref{3.18}: $\eta$ is the coefficient in
$[\partial_n,\tau]$ which is tangent again (certainly, at least, when
applied to $u$ thanks to $\partial_n u=0$ !). Now,
\begin{equation}
  \label{eq:foireux}
B=  \int_{\partial\Omega_S^T \cap M_S^T} w u \, N_2 \cdot \begin{pmatrix}
  \partial_t \\ -\nabla_{\mathrm{tan}} \end{pmatrix} u \,d\sigma_2,
\end{equation}
where $N_2$ is the normal to $\partial\Omega\cap K_S^T$ in
$\partial\Omega$ and $d\sigma_2$ the induced (2D) measure; we leave
$B$ for later treatment.

Adding the first term and the term \eqref{eq:potentieltrace}, we get the
left-handside of \eqref{eq.apriori2}, namely
$$
 \int_{K_S^T\cap
 \partial\Omega_S^T}(|\nabla_\mathrm{tan} u|^2-|\partial_t
 u|^2+\frac{u^6} 3)
 w\,d\sigma dt
$$

We now return to $K_2$, writing (with an obvious abuse of notation for the domains)
\begin{eqnarray*}
  K_2 & = & \int_{D_T\setminus D_S} w (u [Z,\partial_t] u-\partial_t u
 Zu)\, dx\\
 & & {}+\int_{(D_T\setminus D_S) \cap \partial\Omega} w u (N\cdot
 Z)\partial_t u  \,d\sigma\\
 & & {}+\int_{(D_T\setminus D_S) \cap M_S^T} w u (N\cdot
 Z)\partial_t u.
\end{eqnarray*}
The last term vanishes, as $Z$ is tangent to $M_S^T$. The first one is
controlled by $S^2 E_0$. We collect the second one, denoted by $M$, for
later treatment: recalling that on that part of $\partial D$,
$N=-\partial_n$, we have $Z\cdot N=-Z\cdot \partial_n=-1$ and
$$
M=-\int_{(D_T\setminus D_S) \cap \partial\Omega} w u \partial_t u  \,d\sigma.
$$
We now return to $K_1$:
\begin{eqnarray*}
  K_1 & = & \int_{\Omega_S^T\cap M_S^T } w u\,
  [L, Z] u \, d\rho\\
 & & {}+ \int_{\Omega_S^T\cap M_S^T } w u\,
  Z L u \, d\rho.
\end{eqnarray*}
The commutator between two tangent vector fields is tangent, hence the
fist term is controlled through the flux like in \eqref{3.18}. By integration by parts, the
second term is
$$
-\int_{\Omega\cap M_S^T } Z(w u)\,
  L u \, d\rho+\int_{\partial(\Omega\cap M_S^T) } w u\,
  L u (Z\cdot N_1)\, d\sigma_2,
$$
where $N_1$ is the normal to $\Omega\cap M_S^T$ in $M_S^T$, and
$d\sigma_2$ the measure on the boundary. The first term is again
controlled through the flux, and we are left with the second one,
namely
$$
P=\int_{\partial(\Omega\cap M_S^T) } w u\,
  L u (Z\cdot N_1)\, d\sigma_2.
$$
At this point, the only hope is that $B,M$ and $P$ will cancel each
other, as they are integrals over a two-dimensional set which is
the intersection of $M_S^T$ and $\partial\Omega_S^T$. We decompose these
into three distinct regions:
\begin{itemize}
\item on $D_T\cap \partial\Omega$, we get  $ -w u
  \partial_t u$ from $M$ and $w u \partial_t u$ from $B$ as on this
  part, $N_2=\partial_t$. Hence, the total contribution vanishes. The
  same thing (with $N_2=-\partial_t$) applies to $D_S\cap \partial\Omega$.
\item On $(D_T\setminus D_S)\cap M_S^T$, one gets only a contribution
  from $P$: however, on this part of the boundary, $N_1=L$ , and $Z$ is tangent to the cone with no
  time component,
  hence $N_1\cdot Z=0$ and this terms vanishes as well.
\item Finally, we are left with $\partial\Omega\cap M_S^T$. we get non
  zero contributions from $B$ and $P$: both have a factor $w u$, the
  same measure, and the terms (recall for the first one that
  $\partial_n u=0$ !) are equal (respectively) to 
$$
N_2\cdot \begin{pmatrix}\partial_t\\ -\nabla \end{pmatrix} u \text{
  and }(Z\cdot N_1) \nu \cdot \begin{pmatrix}\partial_t\\-\nabla \end{pmatrix} u ,
$$
where, once again, $\nu=\frac{\partial_t+\partial_r} {\sqrt 2}$ is the normal vector
to the cone.

On the (2-dimensional) edge over which the integration is performed,
we denote by $T$ the projection of $Z$: there are two different ways
of writing $Z$ on a direct orthonormal basis on the edge:
\begin{itemize}
\item  we see the edge as the boundary
on $B$ and use 
$$
Z=T+\partial_n+(Z\cdot N_2) N_2\,;
$$
\item we see the edge
as the boundary on $P$ and use 
$$
Z=T+(Z\cdot N_1)
N_1+(Z\cdot \nu) \nu.
$$
\end{itemize}
 From our choice of $Z$,
the very last term in the second decomposition vanishes. On the other
hand, we have a two-dimensional hyperplane where $Z-T$ lives, with two
different basis, $\{-\partial_n,N_2\}$ and $\{N_1,\nu\}$ (where the
$-$ in front of $\partial_n$ results from our choice of $\partial_n$
as the inward normal direction to $\partial\Omega$, while $N_1$ is the
outward normal to $M_S^T\cap \partial\Omega_S^T$). Therefore,
$\nu\cdot N_2=-N_1\cdot \partial_n$. Now, we have
$$
\nu=\lambda \partial_n+\mu N_2, \text{ and } \nu\cdot
\begin{pmatrix}\partial_t\\-\nabla\end{pmatrix} u =\mu  N_2 \cdot
\begin{pmatrix}\partial_t\\ -\nabla\end{pmatrix} u,
$$
where we used (once again !) the Neumann boundary condition. As such,
our two remaining terms compensate exactly if $\mu (Z\cdot N_1)=-1$: but as
\begin{equation*}
  Z-T=(Z\cdot N_1)N_1=\partial_n+(Z\cdot N_2)N_2 \text{ implies } (Z\cdot
  N_1)(N_1\cdot \partial_n)=1,
\end{equation*}
we do get the desired result, namely $(\nu \cdot N_2)(Z\cdot N_1)=-1$.
\end{itemize}
As such, we have disposed with all the boundary terms and this
achieves the proof of Proposition \ref{prop.apriori2}.
\end{proof}
\subsubsection{The $L^6$ estimate}
We are now in position to prove the classical non concentration effect:
\begin{prop}
\label{noconc}
Assume that $x_0\in
  \overline{\Omega}$, and $u$ is a solution to~\eqref{eq.NLW} in the space $X_{<t_0}$, then
\begin{equation}\label{eq.nonconc}
\lim_{t\rightarrow t_0} \int _{D_t} u^6( t,x ) dx=0.
\end{equation}
\end{prop}
\begin{proof} We follow~\cite{St88, Gr90, ShSt93, ShSt94} and simply have
  to take care of the boundary terms. We can assume that $x_0\in
  \partial \Omega$ as otherwise these boundary terms disappear in the
  calculations below (which in this case are standard).  Unlike in~\cite{SmSo95} we cannot use any convexity assumption to obtain
  that these terms have the right sign, but
  Proposition~\ref{prop.apriori2} will serve as a substitute. We may set
  $x_0=0$ and $t_0=0$ for convenience. Integrating over $K_S^T$ the identity
$$ 
0 = \text{div} _{t,x} ( tQ+ u\partial_t u, -tP) + \frac {|u|^6} 3\,,
$$ we get (see~\cite[(3.9)--
  (3.12)]{SmSo95}),
\begin{multline*} 
0 = \int_{D_T}( TQ+ u \partial_t u )(T,x) dx - \int_{D_S} (SQ+ u \partial_t u)(S,x) dx + \frac 1 {\sqrt{2}} \int_{M_S^T} ( tQ+ u \partial_t u + x \cdot P) d\rho\\
- \int_{((S,T) \times \partial \Omega) \cap K_S^T} \nu(x) \cdot (tP) d\sigma + \frac 1 3 \int_{K_S^T} u^6 dx dt\,.
\end{multline*}
Using H\"older's inequality and the conservation of energy, we get
that the first term in the left is controled by $C T E_0$, whereas the last term is
non negative. This yields
\begin{multline} 
 - \int_{D_S} (SQ+ u \partial_t u)(S,x) dx + \frac 1 {\sqrt{2}} \int_{M_S^T} ( tQ+ u \partial_t u + x \cdot P) d\rho\\
\lesssim \int_{((S,0) \times \partial \Omega) \cap K_S^T} \nu(x) \cdot
(tP) d\sigma+ T E_0\,.
\end{multline}
By direct calculation (see~\cite[(3.11)]{SmSo95}),
\begin{multline}
\frac 1 {\sqrt{2}} \int_{M_S^T} ( tQ+ u \partial_t u + x \cdot P) d\rho\\
= \frac 1 {\sqrt{2}} \int_{M_S^T}\frac 1 t | t \partial_t u + x \cdot
\nabla_x u +u|^2 d \rho + \frac 1 2 \int_{\partial D_S}
u^2(S,x) d\sigma- \frac 1 2 \int_{\partial D_T}
u^2(T,x) d\sigma.
\end{multline}
By the trace theorem and H\"older, 
$$
 \int_{\partial D_T}
u^2(T,x) d\sigma\lesssim T \|u(T)\|^2_{H^1(D_T)}\lesssim T E_0.
$$
On the other hand (see~\cite[(3.12)]{SmSo95})
\begin{equation} 
- \int_{D_S} (SQ+ u \partial_t u)(S,x) dx\geq - \frac 1 2 \int _{\partial D_S} u^2 d\sigma- S \int_{D_S} \frac {|u|^6(S,x)} 6 dx\,.
\end{equation}
As a consequence, we obtain
\begin{multline}\label{eq.morawetz}
(-S)\int_{D_S} \frac {|u|^6(S,x)} 6 dx + \frac 1 {\sqrt{2}} \int_{M^T_S} \frac 1 t |t \partial_t u + x\cdot \nabla_x u + u |^2 d\sigma(x,t) \\
\lesssim \int_{\partial\Omega_S^T)\cap K_S^T} n(x) \cdot tP
\,d\sigma dt+T E_0\,,
\end{multline}
Taking~\eqref{eq.P} into account (and the Neumann boundary
condition), we obtain on $\partial \Omega$ 

$$
 tn(x) \cdot P = \frac 1 2 (n(x) \cdot x) \Bigl( \frac {|\partial_t
 u|^2-|\partial_x u|^2}2+\frac{|u|^6} 6\Bigr)  \,.
$$
However, for $x\in \partial \Omega$, given that $x_0=0\in \partial
\Omega$, we have
$$
 \frac{x} {|x|}=T+\mathcal{O}(x),\qquad n(x) = n(0)+ \mathcal{O}(x)
$$
where $T$ is a unit vector tangent to $\partial\Omega$ at $x_0=0$. Consequently, as $n(0) \cdot T=0$,
$$ 
n(x) \cdot x = \mathcal{O}( |x|^2), \text{ for $x\in \partial \Omega$}
$$  
and the right hand side in~\eqref{eq.morawetz} is bounded (using Proposition~\ref{prop.apriori2}) by 
\begin{equation}
 |S|^2 (E_0+E_0^{\frac 2 3})+T E_0.
\end{equation} 
Sending $T$ to zero, the second term disappears, and after dividing by
$-S$, we get
\begin{multline}\label{eq.morawetzbis}
\int_{D_S}|u|^6 \,dx  \lesssim  |S| (E_0+E_0^{\frac 2
  3})+ \frac 1 {\sqrt{2}|S|} \int_{M^0_S} \frac 1 {|t|} |t
\partial_t u + x\cdot \nabla_x u + u |^2 d\sigma(x,t)\,;
\end{multline}
finally, by H\"older's inequality and~\eqref{eq.fluxest}, we obtain
\begin{align*}
\frac 1 {\sqrt{2}|S|} \int_{M^0_S} \frac 1 {|t|} |t \partial_t u +
x\cdot \nabla_x u + u |^2 d\sigma(x,t)  & \leq \sqrt{2} \int_{M^0_S}
\frac{|x|}{|S|} | \frac x {|x|} \partial_t u - \nabla_x u|^2
d\sigma(x,t) \\
 & {}+ \sqrt{2} \int_{M^0_S} \frac {|u|^2}{|S||t|} d
\sigma(x,t)\\
 & \lesssim  \text{ Flux }(u,M_S^0) 
+\text{ Flux }(u,M_S^0)^{1/3}\,,
\end{align*}
hence,
\begin{equation}
  \label{eq.morawetzter}
\int_{D_S}  {|u|^6(S,x)}  dx  \lesssim |S|(E_0+E_0^{\frac 2 3})+ \text{ Flux
  }(u,M_S^0) 
+\text{ Flux }(u,M_S^0)^{1/3}
\end{equation}
which is exactly \eqref{eq.nonconc} thanks to~\eqref{eq.flux}.
Remark that in the calculations above all integrals on $K_S^0$ and $M_S^0$ have to be understood as the limits as $T\rightarrow 0^-$ of the respective integrals on $K_S^T$ and $M_S^T$ (which exist according to \eqref{eq.fluxest}, \eqref{eq.fluxestbis}).

\end{proof}
\subsection{Global existence} 
Once one has obtained the non-concentration result, Proposition~\ref{noconc}, 
 the remaining part
of the proof follows very closely \cite{BLP1} and we reproduce it only
to be self-contained. We consider $u$, the unique forward maximal solution to
the Cauchy problem~\eqref{eq.NLW} in the space $X_{<t_0}$. 
Assume that $t_0 < + \infty$ and consider  a point $x_0 \in
\overline{\Omega}$; our aim is to prove that $u$ can be extended in
a neighborhood of $(x_0,t_0)$, which will imply a contradiction. Up to
a space time translation, we set $(x_0, t_0)= (0,0)$.
\subsubsection{Localizing space-time estimates}
For $t<t'\leq 0$, let us denote by 
$$
\|u\|_{(L^p; L^{q})(K_t^{t'})}= \Bigl(\int_{s=t}^{t'} \Bigl(\int_{\{|x|<-s\} \cap \Omega } |u|^{q}(s,x) dx\Bigr)^{\frac p q} ds\Bigr)^{\frac 1 p}
$$ the $L^p_tL^{q}_x$ norm on $K_t^{t'}$ (with the usual modification
if $p$ or $q$ is infinite). Similarly one defines space-time norms on
boundaries.

Our main result in this section reads
\begin{prop}
\label{prop.L5W}
For any $\varepsilon>0$, there exists $t<0$ such that 
\begin{equation}
\|u\|_{(L^5; L^{10})(K_t^{0})}<\varepsilon.
\end{equation}
\end{prop}
\begin{proof}
We start with an extension result from \cite{BLP1} which applies
equally in our setting:
\begin{lem}\label{lem.extension} For any $x_0\in \overline{\Omega}$
  there exists $r_0>0$ such that for any $0<r<r_0$ and any $v\in
  H^1(\Omega)\cap L^p( \Omega)$, there exist 
a function $\widetilde v_r \in H^1( \Omega)$ (independent of the choice of $1\leq p \leq + \infty$), satisfying
\begin{equation}
\begin{gathered}
(\widetilde v_r -v) _{\mid |x-x_0|<r\cap \Omega } =0,\\
  \int_{\Omega}|\nabla \widetilde v|^2 \lesssim  \int_{ \Omega}|\nabla v|^2, \qquad \|\widetilde v_r\|_{L^p( \Omega)} \lesssim  \|v\|_{L^p( \{|x-x_0|<r\})}.
\end{gathered}
\end{equation}
In other words, we can extend functions in $H^1\cap L^p$ on the ball
$\{|x-x_0|<r\}$ to functions in $H^1(\Omega)\cap L^p( \Omega)$ with
uniform bounds with respect to (small) $r>0$, for the $H^1$ {\em and}
the $L^p$ norms respectively.

 Furthermore, for any $u\in L^\infty((-1,0); H^1_N( \Omega))\cap L^1_{\text{loc}}((-1,0); L^p( \Omega))$, there exist 
a function $\check u \in L^\infty((-1,0); H^1_N( \Omega))\cap L^1_{\text{loc}}((-1,0); L^p( \Omega))$, satisfying
  (uniformly with respect to~$t$)
\begin{equation}\label{eq.time}
\begin{gathered}
(\check u -u) _{\mid \{|x-x_0|<-t\}\cap \Omega} =0,\\
  \int_{\Omega}|\nabla \check u|^2(t, x) + |\partial_t \check u|^2(t,x) dx \lesssim  \int_{\Omega}|\nabla u|^2(t,x) + |\partial_t  u|^2(t,x)\\
 \|\check u(t, \cdot)\|_{L^p( \Omega)} \lesssim  \|u(t, \cdot)\|_{L^p(\Omega \cap  \{|x-x_0|<-t\})}\text{ $t$-a.s. }
\end{gathered}
\end{equation}
\end{lem}
\begin{proof}
See \cite{BLP1} where the boundary condition is easily modified to be
Neumann rather than Dirichlet.
\end{proof}
Let us come back to the proof of Proposition~\ref{prop.L5W}.
 Let $\check u$ be the function given by the second part of Lemma~\ref{lem.extension}.
Then $(\check u)^5 $ is equal to $u^5$ on $K_t^{0}$ and 
\begin{equation}
\begin{aligned}
\|(\check u)^5\|_{L^{\frac 5 4}((t,t'); L^\frac{30} {17}( \Omega)}&\leq \|\check u \|^4_{ L^5((t, t');L^{10}( \Omega))}\|\check u \|_{L^\infty((t,t'); L^6(\Omega)}\\
&\lesssim \|u \|^4_{ (L^5;L^{10})(K_t^{t'}))}\|u \|_{(L^\infty; L^6)(K_t^{0})}.
\end{aligned}
\end{equation}
On the other hand, $\nabla_x (\check u)^5=5 (\check u)^4 \nabla_x \check u$ and 
\begin{equation}
\begin{aligned}
\|\nabla_x(\check u)^5\|_{L^{\frac 5 4}((t,t'); L^\frac{10} {9}( \Omega)}&\leq 5\|\check u \|^4_{ L^5((t, t');L^{10}( \Omega))}\|\nabla_x\check u \|_{L^\infty((t,0); L^2(\Omega)}\\
&\lesssim \| u \|^4_{ (L^5;L^{10})(K_t^{t'})}\|u\|_{L^\infty; H^1(\Omega)}.
\end{aligned}
\end{equation}
By (complex) interpolation, as in ~\eqref{eq.bonstrichartz},
$$ 
\| (\check u)^5\|_{L^{\frac 5 4}((t, {t'}); W^{\frac{7}{10}, \frac {5} 4}( \Omega))}\lesssim \| u \|^4_{ (L^5;L^{10})(K_t^{t'})}\|{u}\|^{\frac 7 {10}}_{L^\infty; H^1(\Omega)}\|u \|^{\frac 3 {10}}_{(L^\infty; L^6)(K_t^{0})}.
$$
Let $w$ be the solution (which, by finite speed of propagation, coincides with $u$ on $K_t^{0}$) of 
$$
 (\partial_s^2 - \Delta )w =-(\check{u})^5,\qquad \partial_n w_{\mid \partial \Omega} =0,\qquad (w-u) _{\mid s=t} =\partial_s( w-u) _{\mid s=t}=0,
$$
applying~\eqref{eq.besov}, and the Sobolev embedding $W^{\frac 3
  {10},5}(\Omega)\mapsto L^{10} ( \Omega)$, we get
\begin{multline}
\label{eq.estim}
  \|u\|_{(L^5;L^{10}) (
  K_t^{t'})} \lesssim (\|w\|_{ (L^5((t, t'); W^{\frac 3 {10},5})(\Omega))}\\
   \lesssim  E (u) +\|u \|^4_{ (L^5; L^{10})(K_t^{t'})}\|u\|^{\frac 7
  {10}}_{L^\infty; H^1(\Omega)}\|u\|^{\frac 3 {10}}_{(L^\infty;
  L^6)(K_t^{0})}\,.
\end{multline}
Finally, from Proposition \ref{noconc}, \eqref{eq.estim} and the continuity of the mapping $
 t'\in [t, 0)\rightarrow \|u\|_{ (L^5; L^{10})(K_t^{t'})} $ (which takes value $0$ for $t'=t$), there
 exists $t$ (close to $0$) such that
$$ \forall t<t'<0\,;\,\,\, \|u\|_{ (L^5; L^{10})(K_t^{t'})}\lesssim  E (u)
$$
and passing to the limit $t' \rightarrow 0$, 
$$\|u\|_{ (L^5; L^{10})(K_t^{0})}\leq 2C E (u) .$$

As a consequence, taking $t<0$ even smaller if necessary, we obtain
\begin{equation}\label{eq.small}
 \|u\|_{ (L^5; L^{10})(K_t^{0})}\leq \varepsilon.
\end{equation}
\end{proof}
 \subsubsection{Global existence}
We are now ready to prove the global existence result.
Let $t<t_0=0$ be close to $0$ and let $v$ be the solution to the linear equation
$$
(\partial_s^2- \Delta) v=0, \qquad \partial_n v_{\mid \partial \Omega} =0,\qquad (v-u)_{\mid s=t} = 0, \qquad \partial_s (v-u) _{\mid s=t}=0,
$$
then the difference $w= u-v$ satisfies 
$$
(\partial_s^2- \Delta) w=- u^5, \qquad \partial_n w_{\mid \partial \Omega} =0, \qquad w_{\mid s=t} = 0, \qquad \partial_s w _{\mid s=t}=0.
$$
Let $\check u$ be the function given by Lemma~\ref{lem.extension} from ${u}$. We have 
$$
\|\check{u} \|_{L^5((t,0);L^{10} ( \Omega))}\lesssim  \varepsilon, \qquad \|\check u \|_{L^\infty; H^1}\lesssim E(u)\,.
$$
Let $\widetilde{w}$ be the solution to
$$
(\partial_s^2- \Delta) \widetilde{w}=- \check{u}^5, \qquad \partial_n \widetilde
{w}_{\mid \partial \Omega} =0, \qquad\widetilde{w}_{\mid s=t} = 0, \qquad
\partial_s \widetilde{w} _{\mid s=t}=0.
$$
By finite speed of propagation, $w$ and $\widetilde{w}$ coincide in $K_t^{0}$. On the other hand, using~\eqref{eq.faible} yields
\begin{multline}
\label{fini}
\|\widetilde{w} \|_{L^\infty((t,0);H^1)} + \|\partial_s \widetilde{w} \|_{L^\infty((t,0); L^2( \Omega))} + \|\widetilde{w} \|_{L^5((t,0); W^{\frac 3 {10},5}( \Omega))}\\
 \lesssim  \|\check{u}^5\|_{L^{1}((t,0); L^2( \Omega)}\lesssim  \|\check{u}\|^5_{L^5((t,0); L^{10} ( \Omega))}\lesssim \varepsilon^5.
\end{multline}
Finally, for any ball $D$, denote by 
$$ 
E(f(s, \cdot), D)= \int_{D\cap \Omega} (|\nabla_x f|^2 + |\partial_s f|^2 + \frac {|f|^6} 3)(s,x) dx;
$$ 
since $v$ is a solution to the linear equation,
\begin{equation}
  \label{eq:lineaire}
  E(v(s, \cdot), D(x_0=0, -s)) \rightarrow 0, \qquad s \rightarrow 0^-
\end{equation}
Recalling that $u=v+\tilde w$ inside $K^0_t$, we obtain from
\eqref{eq:lineaire} and \eqref{fini} (and the Sobolev injection $H^1_N ( \Omega \rightarrow L^6( \Omega)$) that there exists a small $s<0$
such that
$$ 
E(u(s, \cdot), D(x_0=0, -s)) < \varepsilon;
$$
but, since $(u, \partial _s u) (s, \cdot)\in H^1_N( \Omega) \times L^2( \Omega)$, we have,  by dominated convergence,
\begin{multline*}
 E(u(s, \cdot), D(x_0=0, -s))= \int_{\Omega  }1_{ \{|x-x_0|< -s\}} (x)  (|\nabla u( s, x) |^2 + |\partial_s u (s,x)|^2+ \frac {|u|^6(s, x)} 3)) dx\\
\text{ and } \lim_{ \alpha \rightarrow 0}\int_{\Omega} 1_{\{|x-x_0|<\alpha -s\} }(x) (|\nabla u( s, x) |^2 + |\partial_s u (s,x)|^2+ \frac {|u|^6(s,x)} 3)) dx\\
= \lim_{\alpha \rightarrow 0} E(u(s, \cdot), D(x_0=0, -s+\alpha))\,;
\end{multline*}
consequently, there exists $\alpha>0$ such that 
$$ 
E(u(s, \cdot), D(x_0=0, -s+\alpha))\leq 2\varepsilon.
$$
Now, according to~\eqref{eq.locener}, the $L^6$ norm of $u$ remains smaller than $2\varepsilon$ on $\{|x-x_0|<\alpha -s'\},s\leq s' <0$. As a consequence, the same proof as for Proposition~\ref{prop.L5W} shows that the $L^5; L^{10}$ norm of the solution on the truncated cone 
$$ K= \{(x,s'); |x-x_0|<\alpha-s', s<s'<0\}$$
is bounded. Since this is true for all $x_0\in \overline{\Omega}$, a compactness argument shows that 
$$ \|u\|_{L^5((s, 0); L^{10} ( \Omega))} < + \infty$$
which, by Duhamel formula shows that 
$$\lim_{s'\rightarrow 0^-} ( u, \partial_{s} u)(s', \cdot)$$ exists in $(H^1_N ( \Omega) \times L^2( \Omega))$  and consequently $u$ can be extended for $s'>0=t_0$ small enough, using Corollary~\ref{cor.1}.
\begin{figure}[ht]
$$\ecriture{\includegraphics[width=5cm]{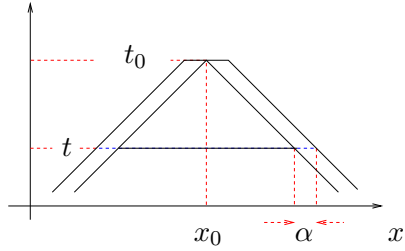}}
{\aat{25}{-2}{$x_0$}\aat{8}{9}{$t$}\aat{50}{-2}{$x$}\aat{38}{-2}{$\alpha$}\aat{16}{21}{$t_0$}}$$ \caption{The truncated cone}\label{fig.cones}
\end{figure}

\section*{Appendix}

To prove Proposition~\ref{propex}, we observe that according to Theorems A and~\ref{th.2}, the operator $\mathcal{T}= e^{\pm it \sqrt{ - \Delta_N}}$satisfies
\begin{equation}\label{eq.strichartz4bis}
\|\mathcal{T}u_0\|_{L^5((0,1) \times \Omega)} \lesssim  \|u_0\|_{H^{\frac 7 {10}}_N( \Omega)}.
\end{equation}
Applying the previous inequality to  $\Delta u_0$, and using the $L^p$ elliptic regularity result 
\begin{equation}
\begin{gathered}
 - \Delta u +u =f \in L^p( \Omega), \quad \partial_n u _{\mid \partial \Omega} =0 \Rightarrow u \in W^{2,p}( \Omega) \cap W^{1,p}_N( \Omega)\\
\text{ and }
\|u\|_{W^{2,p}( \Omega)}\lesssim  \|f\|_{L^p( \Omega)},\qquad 1 <p<+\infty 
\end{gathered}
\end{equation}
we get
\begin{equation}\label{eq.strichartz4}
 \|\mathcal{T}u_0\|_{L^5((0,1);W^{2,5}( \Omega)\cap W^{1,5}_N( \Omega))} \lesssim  \|u_0\|_{H_N^{\frac {27} {10}}( \Omega)}
\end{equation}
and consequently by (complex) interpolation between~\eqref{eq.strichartz4bis} and~\eqref{eq.strichartz4},
\begin{equation}\label{eq.strichartz6} 
 \|\mathcal{T}u_0\|_{L^5((0,1);W^{\frac 3 {10},5}_N (\Omega))} \lesssim  \|u_0\|_{H_N^{1}( \Omega)}\,;
\end{equation} 
finally, by Sobolev embedding
\begin{equation}\label{eq.strichartz8}
 \|\mathcal{T}u_0\|_{L^5((0,1);L^{10}( \Omega))} \lesssim  \|u_0\|_{H_N^{1}( \Omega)}.
\end{equation}
To conclude, we simply observe that
$$
 u= \cos( t\sqrt{ - \Delta_N}) u_0 + \frac { \sin(t\sqrt{ - \Delta_N})}{ \sqrt{ - \Delta_N}} u_1
$$
and $1/\sqrt{- \Delta_N}$ is an isometry from $L^2( \Omega)$ to
$H^1_N( \Omega)$, providing the result.

  The usual $T T^\star$ argument and Christ-Kiselev
  Lemma~\cite{ChKi01} proves the following:

\begin{proof}
We have 
$$
u(t, \cdot)= \cos( t\sqrt{ - \Delta_N})u_0+ \frac{\sin( t\sqrt{ - \Delta_N})} {\sqrt{- \Delta_N}} u_1 + \int_0^t\frac{\sin( (t-s)\sqrt{ - \Delta_N})} {\sqrt{- \Delta_N}}f(s, \cdot) ds.
$$
 The contributions of $(u_0,u_1)$ are easily dealt with, as previously. Let us focus on the contribution of
$$
\int_0^t \frac{e^{i (t-s)\sqrt{ - \Delta_N}}} {\sqrt{- \Delta_N}}f(s, \cdot) ds.
$$ 
Denote by $\mathcal{T}= e^{it \sqrt{-\Delta_N}}$; interpolating between~\eqref{eq.strichartz4bis} and~\eqref{eq.strichartz4}, 
$$
\|\mathcal{T} f\|_{L^5((0,1); W^{2- \frac 7 {10},5}( \Omega)\cap W_N^{1,5}( \Omega))}\lesssim  \|f\|_{H^2_N( \Omega)}.
$$
 Let $u_0\in L^2$. Then there exist $v_0\in H^2_N( \Omega)$ such that 
$$
-\Delta v_0 = u_0, \qquad \|u_0\|_{L^2}\sim \|v_0\|_{H^2};
$$
as a consequence, from $\mathcal{T}u_0= \Delta \mathcal{T}v_0$,
\begin{align*}
 \|\mathcal{T}u_0\|_{L^5((0,1); W^{-\frac 7{10},5}( \Omega))} & \lesssim 
 \|\mathcal{T}v_0\| _{L^5((0,1);W^{2- \frac 7 {10},5}( \Omega))\cap
 W_N^{1,5}( \Omega))}\\
 & \lesssim  \|v_0\|_{H^{2}_N( \Omega)}\sim \|u_0\|_{L^2( \Omega)}.
\end{align*}
By duality we deduce that the operator $\mathcal{T}^*$ defined by
$$
\mathcal{T}^* f= \int_0^1 e^{-is\sqrt{- \Delta}} f(s, \cdot) ds
$$ is bounded from $L^{\frac{5}{4}}((0,1); W^{\frac 7{10},\frac{5}{4}}( \Omega))$ to
$L^2( \Omega)$ (observe that $ W^{\frac 7{10},\frac{5}{4}}( \Omega)=  W_N^{\frac
  7{10},\frac{5}{4}}( \Omega)$); using~\eqref{eq.strichartz6} and boundedness of
 $\sqrt{-\Delta_N}^{-1} $ from $L^2$ to $H^1_N$, we
obtain 
$$
\|{\mathcal{T} {(\sqrt{- \Delta_N})^{-1}}\mathcal{T}^*}
f\|_{L^{5}((0,1); W^{\frac 3{10},5}_N( \Omega))\cap
  L^\infty((0,1); H^1_N( \Omega))}\lesssim \|f\|_{L^{\frac{5}{4}}((0,1); W^{\frac 7{10},\frac{5}{4}}_N(
  \Omega))},
$$
and 
$$
\|\partial_t {\mathcal{T}} (\sqrt{-\Delta_N})^{-1}\mathcal{T}^*f\|_{L^\infty((0,1); L^2( \Omega))}\leq
C\|f\|_{L^{\frac{5}{4}}((0,1); W^{\frac 7{10},\frac{5}{4}}_N(
  \Omega))}.
$$
 But
$${\mathcal{T}} (\sqrt{- \Delta_N})^{-1}\mathcal{T}^*f(s, \cdot)= \int_0^1 \frac{e^{i(t-s)\sqrt{- \Delta_N}}} {\sqrt{- \Delta_N}} f(s, \cdot) ds
$$
and an application of Christ-Kiselev lemma~\cite{ChKi01} allows to transfer this property to the operator
$$
f \mapsto \int_0^t \frac{e^{i(t-s)\sqrt{- \Delta_N}}} {\sqrt{- \Delta_N}} f(s, \cdot) ds.  $$
\end{proof}

\end{document}